\input ./style/arxiv-vmsta.cfg
\documentclass[numbers,compress,v1.0.1]{vmsta}

\volume{2}
\issue{3}
\pubyear{2015}
\firstpage{251}
\lastpage{266}
\doi{10.15559/15-VMSTA37CNF}

\newtheorem{thm}{Theorem}

\newtheorem{cor}{Corollary}
\newtheorem{prop}{Proposition}

\theoremstyle{definition}
\newtheorem{remark}{Remark}

\newtheorem*{xx}{X}

\startlocaldefs
\allowdisplaybreaks
\endlocaldefs

\begin{document}
\begin{frontmatter}

\title{Weak approximation rates for integral functionals of~Markov processes}

\author[a]{\inits{Iu.}\fnm{Iurii}\snm{Ganychenko}\corref{cor1}}\email
{iurii\_ganychenko@ukr.net}
\cortext[cor1]{Corresponding author.}
\author[b]{\inits{A.}\fnm{Alexei}\snm{Kulik}}\email{kulik.alex.m@gmail.com}
\address[a]{Taras Shevchenko National University of Kyiv, Kyiv, Ukraine}
\address[b]{Institute of Mathematics, National Academy of Sciences of
Ukraine, Kyiv, Ukraine}

\markboth{Iu. Ganychenko, A. Kulik}{Weak approximation rates for
integral functionals of Markov processes}

\begin{abstract}
We obtain weak rates for approximation of an integral functional of
a~Markov process by integral sums. An assumption on the process is
formulated only in terms of its transition probability density, and,
therefore, our approach is not strongly dependent on the structure of
the process. Applications to the estimates of the rates of
approximation of the Feynman--Kac semigroup and of the price of
``occupation-time options'' are provided.
\end{abstract}


\received{6 September 2015}
\accepted{14 September 2015}
\publishedonline{23 September 2015}
\end{frontmatter}

\section{Introduction and main results}

Let $X_t, t\geq0$, be a Markov process with values in $\mathbb{R}^d$.
We consider the following objects:

1) the integral functional
\[
I_T(h)=\int_0^Th(X_t)
\, dt
\]
of this process;

2) the sequence of integral sums
\[
I_{T,n}(h)={T\over n}\sum_{k=0}^{n-1}h(X_{(kT)/n}),
\quad n\geq1.
\]
The problem we are focused on is obtaining upper bounds on the accuracy
of approximation of the integral functional $I_{T}(h)$ by the integral
sums $I_{T,n}(h)$ without any regularity assumption on the function
$h$. The function $h$ is only assumed to be measurable and bounded.
Therefore, the class of functionals $I_T(h)$ contains, for example, the
\emph{occupation time} of the process $X$ in a set $A\subset\mathcal
{B}(\mathbb{R}^d)$ (in this case, $h=\mathbb{I}_A$).

The problem of estimating the expectation of expressions that contain
both the value of a process and the value of an integral functional of
this process arises naturally in a wide range of probabilistic
problems. Two of them related with the\break Feynman--Kac semigroup and the
price of an occupation-time option are discussed in Section~\ref{s3}.
An exact calculation of such expressions, if possible, can be performed
only under substantial assumptions on the structure of functionals and
processes; see, for example, \cite{Guerin}, where the price of an
occupation-time option is precisely calculated for a L\'{e}vy process
with only negative jumps. For more complicated models, it is natural to
use approximative methods, which naturally require estimates of
approximation errors. This motivates the main aim of the paper to
evaluate the error bounds for discrete approximations of the integral
functional $I_T(h)$.

In what follows, $P_x$ denotes the law of the Markov process $X$
conditioned by $X_0=x$, and $\mathbb{E}_x$ denotes the expectation with
respect to this law. Both the absolute value of a real number and the
Euclidean norm in $\mathbb{R}^d$ are denoted by $|\cdot|$; $\|\cdot\|$
denotes the $\sup$-norm in $L_\infty$.

The following result was obtained in \cite{Gobet} as a part of the
proof of a more general statement (see Theorem 2.5 in \cite{Gobet}).
\begin{prop}\label{tv1}
Suppose that $X$ is a multidimensional diffusion process with bounded
H\"{o}lder continuous coefficients and that its diffusion coefficient
satisfies the uniform ellipticity condition
\[
\bigl(a(x)\theta,\theta\bigr)_{\mathbb{R}^d} \geq c |\theta|^2, \quad x,
\theta\in \mathbb{R}^d, \ c > 0.
\]

Then there exists a positive constant $C$ such that
\begin{equation}
\label{lem} \big| E_x I_{T}(h) - E_x
I_{T,n}(h) \big| \leq C\|h\| \frac{\log n}{n}.
\end{equation}
\end{prop}

The scheme of the proof of this result can be extended
straightforwardly to the case of arbitrary Markov process that
satisfies the following assumption (see Proposition 2.1 \cite{kul-gan}):

\begin{xx}
The process $X$ possesses a transition probability density
$p_t(x,y)$ that is differentiable with respect to $t$ and satisfies
\begin{equation}
\label{der_bound} \big|\partial_tp_t(x,y) \big|\leq
C_Tt^{-1} q_{t,x}(y), \quad t\leq T, \
C_T \geq1,
\end{equation}
for some measurable function $q$ such that for any fixed $t$ and $x$,
the function $q_{t,x}$ is a distribution density.
\end{xx}

\begin{remark}
A diffusion process satisfies condition \textbf{X} with
\[
q_{t,x}(y) \ = c_1 t^{-d/2} \exp
\bigl(-c_2t^{-1}|x-y|^2\bigr)
\]
and properly chosen $c_1, c_2$. The other examples of the processes
satisfying condition \eqref{der_bound} are provided in {\rm\cite
{kul-gan}}. Among them, we should mention an $\alpha$-stable process.
\end{remark}

Under assumption \textbf{X}, Proposition~\ref{tv1} and Proposition 2.1
in \cite{kul-gan} give bounds for the rate of approximation of \emph
{expectations} of the integral functionals of the process $X$. Such
approximation rates are called \emph{weak}. \emph{Strong} $L_p$-rates,
that is, the bounds for
\[
E_x \big|I_{T}(h)-I_{T,n}(h) \big|^p,
\]
have been recently obtained in \cite{Kohatsu-Higa} for diffusion
processes and in \cite{kul-gan} without restrictions on the structure
of the processes. In this paper, we provide another generalization of
the weak rate (\ref{lem}), namely, the rates of approximation for
expectations of more complicated functionals.
Let us formulate the main result of this paper.

\begin{thm}\label{thm1} Suppose that \textbf{X} holds. Then for each $k
\in\mathbb{N}$ and any bounded function~$f$,
\[
\big|E_x \bigl(I_{T}(h)\bigr)^k
f(X_T)- E_x\bigl(I_{T,n}(h)
\bigr)^kf(X_T) \big| \leq6k^2 C_T
T^{k} \|h\|^k \biggl({\log{n}\over n} \biggr) \|f
\|.
\]
\end{thm}

Clearly, Proposition 2.1 in \cite{kul-gan} is a particular case of
Theorem~\ref{thm1}. The latter statement is a substantial extension of
the former one: it contains both the moments of any order of the
integral functional and the value of the process in the final time
moment. Using the Taylor expansion, we obtain the following corollary
of Theorem~\ref{thm1}.

Consider any analytic function $g$ defined in a neighborhood of 0 and
constants $D_g,R_g>0$ such that $ |\frac{g^{(m)}(0)}{m!} | \leq
D_g  ( \frac{1}{R_g}  )^m$ for any natural $m$.

\begin{cor}\label{n1}
Suppose that \textbf{X} holds. Then for any bounded function $f$ and
a~function $h$ such that $T\|h\|< R_g$, we have:
\begin{equation}
\label{nas} \big|E_x g\bigl( I_{T}(h)\bigr)
f(X_T)- E_x g\bigl( I_{T,n}(h)
\bigr)f(X_T) \big| \leq C_{T,h,D_g,R_g} \biggl({\log{n}\over n}
\biggr) \|f\|,
\end{equation}
where
\[
C_{T,h,D_g,R_g}=6D_gC_T\frac{ T\|h\|}{R_g} \biggl(1+
\frac{ T\|h\|
}{R_g} \biggr) \frac{1}{ (1-\frac{ T\|h\|}{R_g} )^3}.
\]
\end{cor}

We provide the proof of Theorem~\ref{thm1} in Section~\ref{s2}. In
Section~\ref{s3}, we give an~application of Theorem~\ref{thm1} to
estimates of the rates of approximation of the Feynman--Kac semigroup
and of the price of an occupation-time option.

\section{Proof of Theorem~\ref{thm1}}\label{s2}

Denote
\[
S_{k,a,b}:= \bigl\{(s_1,s_2,\ldots,s_k)
\in\mathbb{R}^k | a\leq s_1 < s_2 < \cdots<
s_k \leq b\bigr\}, \quad k \in\mathbb{N}, \ a,b \in\mathbb{R},
\]
and for each $t\in[kT/n, (k+1)T/n)$, put $\eta_n(t):={kT\over
n}$.\vadjust{\eject}

We have:
\begin{align*}
&\frac{1}{k!} \bigl(E_x \bigl[\bigl(I_{T}(h)\bigr)^k-\bigl(I_{T,n}(h)\bigr)^k \bigr]f(X_T) \bigr)\\
&\quad = E_x \int_{S_{k,0,T}} \Biggl[\prod_{i=1}^k h(X_{s_i}) - \prod_{i=1}^k h(X_{\eta_n(s_i)}) \Biggr]f(X_T) \prod_{i=1}^k ds_{i}\\
&\quad = \int_{S_{k,0,T}}\int_{(\mathbb{R}^d)^{k+1}} \Biggl(\prod_{i=1}^kh(y_i) \Biggr)f(z)\Biggl(\prod_{i=1}^k p_{s_i-s_{i-1}}(y_{i-1},y_i)\Biggr)p_{T-s_k}(y_k,z)\\
&\qquad \times dz \prod_{j=1}^k dy_j\prod_{i=1}^k ds_{i} - \int_{S_{k,0,T}}\int_{(\mathbb{R}^d)^{k+1}} \Biggl(\prod_{i=1}^kh(y_i) \Biggr)f(z)\\
&\qquad \times \Biggl(\prod_{i=1}^k p_{\eta_n(s_i)-\eta_n(s_{i-1})}(y_{i-1},y_i) \Biggr) p_{T-\eta_n(s_k)}(y_k,z)dz \prod_{j=1}^k dy_j \prod_{i=1}^k ds_{i},
\end{align*}
where $s_0 = 0, \ y_0 = x$.

Rewrite the expression under the integral
\begin{align*}
&\Biggl(\prod_{i=1}^k p_{s_i-s_{i-1}}(y_{i-1},y_i)\Biggr) p_{T-s_k}(y_k,z)\\[-2pt]
&\quad - \Biggl( \prod_{i=1}^k
p_{\eta_n(s_i)-\eta_n(s_{i-1})}(y_{i-1},y_i) \Biggr) p_{T-\eta_n(s_k)}(y_k,z)
\end{align*}
in the form
\begin{align*}
&\Biggl(\prod_{i=1}^k p_{s_i-s_{i-1}}(y_{i-1},y_i)\Biggr) p_{T-s_k}(y_k,z)\\[-2pt]
&\qquad \mp p_{\eta_n(s_1)}(x,y_1) \Biggl(\prod_{i=2}^k p_{s_i-s_{i-1}}(y_{i-1},y_i)\Biggr)p_{T-s_k}(y_k,z)\\[-2pt]
&\qquad - \Biggl( \prod_{i=1}^k p_{\eta_n(s_i)-\eta_n(s_{i-1})}(y_{i-1},y_i) \Biggr) p_{T-\eta_n(s_k)}(y_k,z)\\[-2pt]
&\quad = \Biggl(\prod_{i=1}^k p_{s_i-s_{i-1}}(y_{i-1},y_i) \Biggr) p_{T-s_k}(y_k,z)\\[-2pt]
&\qquad \mp p_{\eta_n(s_1)}(x,y_1) \Biggl(\prod_{i=2}^k p_{s_i-s_{i-1}}(y_{i-1},y_i)\Biggr)p_{T-s_k}(y_k,z)\\[-2pt]
&\qquad \mp p_{\eta_n(s_1)}(x,y_1) p_{s_2-\eta_n(s_1)}(y_{1},y_2)\Biggl(\prod_{i=3}^k p_{s_i-s_{i-1}}(y_{i-1},y_i)\Biggr)p_{T-s_k}(y_k,z)\\[-2pt]
&\qquad - \Biggl( \prod_{i=1}^k p_{\eta_n(s_i)-\eta_n(s_{i-1})}(y_{i-1},y_i) \Biggr) p_{T-\eta_n(s_k)}(y_k,z)\\[-2pt]
&\quad \qquad \qquad \qquad \qquad \qquad \qquad \vdots\\[-2pt]
&\quad = J_1+J_2+\cdots+J_{2k-1}+J_{2k},
\end{align*}
where
\begin{align*}
J_1 &\,{=}\, \bigl(p_{s_1}(x,y_1)-p_{\eta_n(s_1)}(x,y_1)\bigr) \Biggl(\prod_{i=2}^k p_{s_i-s_{i-1}}(y_{i-1},y_i) \Biggr)p_{T-s_k}(y_k,z),\\
J_2 &\,{=}\, p_{\eta_n(s_1)}(x,y_1) \bigl(p_{s_2-s_1}(y_{1},y_2)- p_{s_2-\eta_n(s_1)}(y_{1},y_2)\bigr)\\
&\quad \times \Biggl(\prod_{i=3}^k p_{s_i-s_{i-1}}(y_{i-1},y_i) \Biggr)p_{T-s_k}(y_k,z),\\
J_3 &\,{=}\, p_{\eta_n(s_1)}(x,y_1) \bigl(p_{s_2-\eta_n(s_1)}(y_{1},y_2)- p_{\eta_n(s_2)-\eta_n(s_1)}(y_{1},y_2)\bigr)\\
&\quad \times \Biggl(\prod_{i=3}^k p_{s_i-s_{i-1}}(y_{i-1},y_i)\Biggr)p_{T-s_k}(y_k,z),\\
J_4 &\,{=}\, p_{\eta_n(s_1)}(x,y_1)p_{\eta_n(s_2)-\eta_n(s_1)}(y_{1},y_2)\bigl(p_{s_3-s_2}(y_{2},y_3) - p_{s_3-\eta_n(s_2)}(y_{2},y_3)\bigr)\\
&\quad \times \Biggl(\prod_{i=4}^k p_{s_i-s_{i-1}}(y_{i-1},y_i) \Biggr)p_{T-s_k}(y_k,z),\\
J_5 &\,{=}\, p_{\eta_n(s_1)}(x,y_1)p_{\eta_n(s_2)-\eta_n(s_1)}(y_{1},y_2)\bigl(p_{s_3-\eta_n(s_2)}(y_{2},y_3) \,{-} p_{\eta_n(s_3)-\eta_n(s_2)}(y_{2},y_3)\bigr)\\
&\quad \times \Biggl(\prod_{i=4}^k p_{s_i-s_{i-1}}(y_{i-1},y_i) \Biggr)p_{T-s_k}(y_k,z),\\
&\quad \qquad \qquad \qquad \qquad \qquad \qquad \vdots\\
J_{2k-1} &\,{=}\, \Biggl(\prod_{i=1}^{k-1}p_{\eta_n(s_i)-\eta_n(s_{i-1})}(y_{i-1},y_i) \Biggr)\\
&\quad \times\bigl(p_{s_k-\eta_n(s_{k-1})}(y_{k-1},y_k) - p_{\eta_n(s_k)-\eta_n(s_{k-1})}(y_{k-1},y_k)\bigr)p_{T-s_k}(y_k,z),\\
J_{2k} &\,{=}\, \Biggl(\prod_{i=1}^{k}p_{\eta_n(s_i)-\eta_n(s_{i-1})}(y_{i-1},y_i) \Biggr) \bigl(p_{T-s_k}(y_k,z)-p_{T-\eta_n(s_k)}(y_k,z)\bigr).
\end{align*}

Therefore,
\begin{align}
&\frac{1}{k!} \bigl(E_x \bigl[\bigl(I_{T}(h)\bigr)^k-\bigl(I_{T,n}(h)\bigr)^k \bigr]f(X_T) \bigr)\nonumber\\
&\!\quad = \!\!\int_{S_{k,0,T}}\int_{(\mathbb{R}^d)^{k+1}}\Biggl(\prod_{i=1}^kh(y_i)\Biggr)f(z) (J_1\,{+}\,J_2\,{+}\,\cdots\,{+}\,J_{2k-1}\,{+}\,J_{2k})dz\prod_{j=1}^k dy_j \prod_{i=1}^k ds_{i}.\label{sumJi}
\end{align}

Our way to estimate each of $2k$ terms in (\ref{sumJi}) is mostly the
same, but its realization is different for the first, the last, and the
intermediate terms. Let us estimate the first term in
(\ref{sumJi}):\vadjust{\eject}
\begin{align*}
&\Bigg|\int_{S_{k,0,T}}\int_{(\mathbb{R}^d)^{k+1}} \Biggl(\prod_{i=1}^kh(y_i) \Biggr)f(z)J_1 dz \prod_{j=1}^k dy_j \prod_{i=1}^k ds_{i} \Bigg|\\
&\quad = \Bigg|\int_0^T \int_{S_{k-1,s_1,T}}\int_{(\mathbb{R}^d)^{k+1}} \Biggl(\prod_{i=1}^kh(y_i)\Biggr)f(z) J_1 dz \prod_{j=1}^k dy_j \prod_{i=2}^k ds_{i} ds_1 \Bigg|\\
&\quad \leq \Bigg|\int_0^{T/n} \int_{S_{k-1,s_1,T}}\int_{(\mathbb{R}^d)^{k+1}} \Biggl(\prod_{i=1}^kh(y_i)\Biggr)f(z) J_1 dz \prod_{j=1}^k dy_j \prod_{i=2}^k ds_{i} ds_1 \Bigg|\\
&\qquad + \Bigg|\int_{T/n}^T \int_{S_{k-1,s_1,T}}\int_{(\mathbb{R}^d)^{k+1}} \Biggl(\prod_{i=1}^kh(y_i)\Biggr)f(z) J_1 dz \prod_{j=1}^k dy_j \prod_{i=2}^k ds_{i} ds_1 \Bigg|
\end{align*}
Let us consider each term in detail:
\begin{align*}
&\Bigg|\int_0^{T/n} \int_{S_{k-1,s_1,T}}\int_{(\mathbb{R}^d)^{k+1}} \Biggl(\prod_{i=1}^kh(y_i)\Biggr)f(z) J_1 dz \prod_{j=1}^k dy_j \prod_{i=2}^k ds_{i} ds_1 \Bigg|\\
&\quad \leq\int_0^{T/n} \int_{S_{k-1,0,T}}\int_{(\mathbb{R}^d)^{k+1}} \Bigg| \Biggl(\prod_{i=1}^kh(y_i)\Biggr)f(z) J_1 \Bigg| dz \prod_{j=1}^k dy_j \prod_{i=2}^k ds_{i} ds_1\\
&\quad \leq\|h\|^k\|f\| \int_0^{T/n} \int_{S_{k-1,0,T}}\int_{(\mathbb{R}^d)^{k+1}}|J_1| dz \prod_{j=1}^k dy_j \prod_{i=2}^k ds_{i} ds_1\\
&\quad \leq\frac{1}{(k-1)!} \|h\|^k\|f\| T^{k-1} \int_0^{T/n} \int_{\mathbb{R}^d}\big|p_{s_1}(x,y_1)-p_{\eta_n(s_1)}(x,y_1)\big|dy_1 ds_1\\
&\quad \leq\frac{2}{(k-1)!} \|h\|^k\|f\| T^{k}\frac{1}{n}.
\end{align*}
Next, we have
\begin{align*}
&\Bigg|\int_{T/n}^T \int_{S_{k-1,s_1,T}}\int_{(\mathbb{R}^d)^{k+1}} \Biggl(\prod_{i=1}^kh(y_i)\Biggr)f(z) J_1 dz \prod_{j=1}^k dy_j \prod_{i=2}^k ds_{i} ds_1 \Bigg|\\
&\quad = \Bigg|\int_{T/n}^T \int_{\eta_n(s_1)}^{s_1}\int_{S_{k-1,s_1,T}}\int_{(\mathbb{R}^d)^{k+1}} \Biggl(\prod_{i=1}^kh(y_i) \Biggr)f(z)\partial_u p_u (x,y_1)\\
&\qquad \times \Biggl(\prod_{i=2}^k p_{s_i-s_{i-1}}(y_{i-1},y_i) \Biggr) p_{T-s_k}(y_k,z)dz \prod_{j=1}^k dy_j \prod_{i=2}^k ds_{i} du ds_1 \Bigg|\\
&\quad \leq\int_{T/n}^T \int_{\eta_n(s_1)}^{s_1}\int_{S_{k-1,0,T}}\int_{(\mathbb{R}^d)^{k+1}} \Bigg| \Biggl(\prod_{i=1}^kh(y_i) \Biggr)f(z)\partial_u p_u (x,y_1)\\
&\qquad \times \Biggl(\prod_{i=2}^k p_{s_i-s_{i-1}}(y_{i-1},y_i) \Biggr) p_{T-s_k}(y_k,z)\Bigg| dz \prod_{j=1}^k dy_j \prod_{i=2}^k ds_{i}\times du ds_1\\
&\quad \leq\|h\|^k\|f\| \int_{T/n}^T \int_{\eta_n(s_1)}^{s_1} \int_{S_{k-1,0,T}}\int_{(\mathbb{R}^d)^{k+1}}\big|\partial_u p_u(x,y_1)\big| \Biggl(\prod_{i=2}^k p_{s_i-s_{i-1}}(y_{i-1},y_i) \Biggr)\\
&\qquad \times p_{T-s_k}(y_k,z) dz \prod_{j=1}^k dy_j \prod_{i=2}^k ds_{i} du ds_1.
\end{align*}

Integrating over $z, y_k, y_{k-1}, \dots, y_2$ and then over $s_k,
s_{k-1}, \dots, s_2$, we derive
\begin{align*}
&\Bigg|\int_{T/n}^T \int_{S_{k-1,s_1,T}}\int_{(\mathbb{R}^d)^{k+1}} \Biggl(\prod_{i=1}^kh(y_i)\Biggr)f(z) J_1 dz \prod_{j=1}^k dy_j \prod_{i=2}^k ds_{i} ds_1 \Bigg|\\
&\quad \leq\frac{1}{(k-1)!} \|h\|^k\|f\| T^{k-1} \int_{T/n}^T \int_{\eta_n(s_1)}^{s_1}\int_{\mathbb{R}^d}\big|\partial_u p_u(x,y_1)\big| dy_1 du ds_1\\
&\quad \leq C_T \frac{1}{(k-1)!} \|h\|^k\|f\| T^{k-1}\int_{T/n}^T \int_{\eta_n(s_1)}^{s_1}\int_{\mathbb{R}^d}u^{-1} q_{u,x}(y_1)dy_1 du ds_1\\
&\quad = C_T \frac{1}{(k-1)!} \|h\|^k\|f\| T^{k-1}\int_{T/n}^T \int_{\eta_n(s_1)}^{s_1}u^{-1} du ds_1\\
&\quad = C_T \frac{1}{(k-1)!} \|h\|^k\|f\| T^{k-1}\sum_{i=1}^{n-1} \int_{iT/n}^{(i+1)T/n}\int_{iT/n}^{s_1} u^{-1} du ds_1\\
&\quad = C_T \frac{1}{(k-1)!} \|h\|^k\|f\| T^{k-1}\sum_{i=1}^{n-1} \int_{iT/n}^{(i+1)T/n}\int_{u}^{(i+1)T/n} u^{-1} ds_1 du\\
&\quad \leq C_T \frac{1}{(k-1)!} \|h\|^k\|f\| T^{k}\frac{1}{n} \sum_{i=1}^{n-1} \int_{iT/n}^{(i+1)T/n} u^{-1} du\\
&\quad = C_T \frac{1}{(k-1)!} \|h\|^k\|f\| T^{k}\frac{1}{n} \int_{T/n}^{T} u^{-1} du\\
&\quad = C_T \frac{1}{(k-1)!} \|h\|^k\|f\| T^{k}\frac{\log n}{n}.
\end{align*}
Therefore,
\begin{align*}
&\Bigg|\int_{S_{k,0,T}}\int_{(\mathbb{R}^d)^{k+1}} \Biggl(\prod_{i=1}^kh(y_i) \Biggr)f(z)J_1 dz \prod_{j=1}^k dy_j \prod_{i=1}^k ds_{i} \Bigg|\\
&\quad \leq3 C_T \frac{1}{(k-1)!} \|h\|^k\|f\|T^{k} \frac{\log n}{n}.
\end{align*}
Now we are ready to estimate the last summand in (\ref{sumJi}):
\begin{align*}
&\Bigg|\int_{S_{k,0,T}}\int_{(\mathbb{R}^d)^{k+1}} \Biggl(\prod_{i=1}^kh(y_i) \Biggr)f(z)J_{2k} dz \prod_{j=1}^k dy_j \prod_{i=1}^k ds_{i} \Bigg|\\
&\quad = \Bigg|\int_0^T \int_{S_{k-1,0,s_k}}\int_{(\mathbb{R}^d)^{k+1}} \Biggl(\prod_{i=1}^kh(y_i)\Biggr)f(z) J_{2k} dz \prod_{j=1}^k dy_j \prod_{i=1}^{k-1} ds_{i} ds_k \Bigg|\\
&\quad \leq \Bigg|\int_0^{T-T/n} \int_{S_{k-1,0,s_k}}\int_{(\mathbb{R}^d)^{k+1}} \Biggl(\prod_{i=1}^kh(y_i)\Biggr)f(z) J_{2k} dz \prod_{j=1}^k dy_j \prod_{i=1}^{k-1} ds_{i} ds_k \Bigg|\\
&\qquad + \Bigg|\int_{T-T/n}^T \int_{S_{k-1,0,s_k}}\int_{(\mathbb{R}^d)^{k+1}} \Biggl(\prod_{i=1}^kh(y_i)\Biggr)f(z) J_{2k} dz \prod_{j=1}^k dy_j \prod_{i=1}^{k-1}ds_{i} ds_k \Bigg|
\end{align*}
Let us estimate each term separately. We get
\begin{align*}
&\Bigg|\int_{T-T/n}^T \int_{S_{k-1,0,s_k}}\int_{(\mathbb{R}^d)^{k+1}} \Biggl(\prod_{i=1}^kh(y_i)\Biggr)f(z) J_{2k} dz \prod_{j=1}^k dy_j \prod_{i=1}^{k-1}ds_{i} ds_k \Bigg|\\
&\quad \leq\int_{T-T/n}^T \int_{S_{k-1,0,T}}\int_{(\mathbb{R}^d)^{k+1}} \Bigg| \Biggl(\prod_{i=1}^kh(y_i)\Biggr)f(z) J_{2k} \Bigg| dz \prod_{j=1}^k dy_j \prod_{i=1}^{k-1}ds_{i} ds_k\\
&\quad \leq\|h\|^k\|f\| \int_{T-T/n}^T \int_{S_{k-1,0,T}}\int_{(\mathbb{R}^d)^{k+1}}|J_{2k}| dz \prod_{j=1}^k dy_j \prod_{i=1}^{k-1} ds_{i} ds_k\\
&\quad \leq\frac{2}{(k-1)!} \|h\|^k\|f\| T^{k}\frac{1}{n}.
\end{align*}
For the other term, we obtain:
\begin{align*}
&\Bigg|\int_0^{T-T/n} \int_{S_{k-1,0,s_k}}\int_{(\mathbb{R}^d)^{k+1}} \Biggl(\prod_{i=1}^kh(y_i)\Biggr)f(z) J_{2k} dz \prod_{j=1}^k dy_j \prod_{i=1}^{k-1}ds_{i} ds_k \Bigg|\\
&\quad = \Bigg|\int_0^{T-T/n} \int_{\eta_n(s_k)}^{s_k}\int_{S_{k-1,0,s_k}}\int_{(\mathbb{R}^d)^{k+1}} \Biggl(\prod_{i=1}^kh(y_i) \Biggr)f(z)\\
&\qquad \times \Biggl(\prod_{i=1}^{k}p_{\eta_n(s_i)-\eta_n(s_{i-1})}(y_{i-1},y_i) \Biggr) \partial_{T-u}p_{T-u} (y_k,z) dz \prod_{j=1}^k dy_j \prod_{i=1}^{k-1}ds_{i} du ds_k \Bigg|\\
&\quad \leq\int_0^{T-T/n} \int_{\eta_n(s_k)}^{s_k}\int_{S_{k-1,0,T}}\int_{(\mathbb{R}^d)^{k+1}} \Bigg| \Biggl(\prod_{i=1}^kh(y_i) \Biggr)f(z)\\
&\qquad \times \Biggl(\prod_{i=1}^{k}p_{\eta_n(s_i)-\eta_n(s_{i-1})}(y_{i-1},y_i) \Biggr) \partial_{T-u}p_{T-u} (y_k,z) \Bigg| dz \prod_{j=1}^k dy_j \prod_{i=1}^{k-1}ds_{i} du ds_k\\
&\quad \leq\|h\|^k\|f\| \int_0^{T-T/n} \int_{\eta_n(s_k)}^{s_k} \int_{S_{k-1,0,T}}\int_{(\mathbb{R}^d)^{k+1}} \Biggl(\prod_{i=1}^{k}p_{\eta_n(s_i)-\eta_n(s_{i-1})}(y_{i-1},y_i) \Biggr)\\
&\qquad \times\big|\partial_{T-u} p_{T-u} (y_k,z)\big| dz \prod_{j=1}^k dy_j \prod_{i=1}^{k-1} ds_{i} du ds_k.
\end{align*}

Let us rewrite this expression in the form\vspace{-3pt}
\begin{align*}
&\|h\|^k\|f\| \int_{S_{k-1,0,T}}\int_{(\mathbb{R}^d)^{k-1}}\Biggl(\prod_{i=1}^{k-1} p_{\eta_n(s_i)-\eta_n(s_{i-1})}(y_{i-1},y_i)\Biggr)\\
&\quad \times\int_0^{T-T/n} \int_{\eta_n(s_k)}^{s_k}\int_{(\mathbb{R}^d)^2} p_{\eta_n(s_k)-\eta_n(s_{k-1})}(y_{k-1},y_k)\\
&\quad \times\big|\partial_{T-u} p_{T-u} (y_k,z)\big| dz dy_k du ds_k \prod_{j=1}^{k-1}dy_j \prod_{i=1}^{k-1}ds_{i}
\end{align*}
and consider the inner integral
\begin{align*}
&\int_0^{T-T/n} \int_{\eta_n(s_k)}^{s_k}\int_{(\mathbb{R}^d)^2} p_{\eta_n(s_k)-\eta_n(s_{k-1})}(y_{k-1},y_k)\big|\partial_{T-u} p_{T-u} (y_k,z)\big| dz dy_k du ds_k\\
&\quad =\sum_{i=0}^{n-2} \int_{iT/n}^{(i+1)T/n}\int_{iT/n}^{s_k} \int_{(\mathbb{R}^d)^2}p_{iT/n-\eta_n(s_{k-1})}(y_{k-1},y_k)\\[1pt]
&\qquad \times\big|\partial_{T-u} p_{T-u} (y_k,z)\big| dz dy_k du ds_k\\[1pt]
&\quad =\sum_{i=0}^{n-2} \int_{iT/n}^{(i+1)T/n}\int_{u}^{(i+1)T/n} \int_{(\mathbb{R}^d)^2}p_{iT/n-\eta_n(s_{k-1})}(y_{k-1},y_k)\\[1pt]
&\qquad \times\big|\partial_{T-u} p_{T-u} (y_k,z)\big| dz dy_k ds_k du\\[1pt]
&\quad \leq\frac{T}{n} \sum_{i=0}^{n-2}\int_{iT/n}^{(i+1)T/n} \int_{(\mathbb{R}^d)^2}p_{iT/n-\eta_n(s_{k-1})}(y_{k-1},y_k)\big|\partial_{T-u}p_{T-u} (y_k,z)\big| dz dy_k du\\
&\quad \leq C_T \frac{T}{n} \sum_{i=0}^{n-2}\int_{iT/n}^{(i+1)T/n} \int_{\mathbb{R}^d}p_{iT/n-\eta_n(s_{k-1})}(y_{k-1},y_k) (T-u)^{-1}dy_k du\\
&\quad = C_T \frac{T}{n} \sum_{i=0}^{n-2}\int_{iT/n}^{(i+1)T/n} (T\,{-}\,u)^{-1} du \,{=}\, C_T \frac{T}{n} \int_{0}^{T-T/n}(T\,{-}\,u)^{-1} du \,{=}\, TC_T \frac{\log n}{n}.
\end{align*}
Therefore, we have:
\begin{align*}
&\Bigg|\int_0^{T-T/n} \int_{S_{k-1,0,s_k}}\int_{(\mathbb{R}^d)^{k+1}} \Biggl(\prod_{i=1}^kh(y_i)\Biggr)f(z) J_{2k} dz \prod_{j=1}^k dy_j \prod_{i=1}^{k-1}ds_{i} ds_k \Bigg|\\
&\quad \leq C_T \frac{1}{(k-1)!} \|h\|^k\|f\| T^{k}\frac{\log n}{n}
\end{align*}
and\vspace{-3pt}
\begin{align*}
&\Bigg|\int_{S_{k,0,T}}\int_{(\mathbb{R}^d)^{k+1}} \Biggl(\prod_{i=1}^kh(y_i) \Biggr)f(z)J_{2k} dz \prod_{j=1}^k dy_j \prod_{i=1}^k ds_{i} \Bigg|\\
&\quad \leq3 C_T \frac{1}{(k-1)!} \|h\|^k\|f\|T^{k} \frac{\log n}{n}.
\end{align*}
To complete the proof, we should additionally consider the following
terms in~(\ref{sumJi}):
\begin{align}
&\Bigg|\int_{S_{k,0,T}}\int_{(\mathbb{R}^d)^{k+1}} \Biggl(\prod_{i=1}^kh(y_i) \Biggr)f(z)\Biggl(\prod_{l=1}^{j-1} p_{\eta_n(s_l)-\eta_n(s_{l-1})}(y_{l-1},y_l)\Biggr)\nonumber\\
&\quad \times\bigl(p_{s_j-s_{j-1}}(y_{j-1},y_j)-p_{s_j-\eta_n(s_{j-1})}(y_{j-1},y_j)\bigr)\nonumber\\
&\quad \times \Biggl(\prod_{m={j+1}}^k p_{s_m-s_{m-1}}(y_{m-1},y_m) \Biggr)p_{T-s_k}(y_k,z)dz \prod_{q=1}^k dy_q \prod_{r=1}^k ds_{r} \Bigg|\label{j}
\end{align}
and
\begin{align}
&\Bigg|\int_{S_{k,0,T}}\int_{(\mathbb{R}^d)^{k+1}} \Biggl(\prod_{i=1}^kh(y_i) \Biggr)f(z)\Biggl(\prod_{l=1}^{j-1} p_{\eta_n(s_l)-\eta_n(s_{l-1})}(y_{l-1},y_l)\Biggr)\nonumber\\
&\quad \times\bigl(p_{s_j-\eta_n(s_{j-1})}(y_{j-1},y_j) - p_{\eta_n(s_j)-\eta_n(s_{j-1})}(y_{j-1},y_j)\bigr)\nonumber\\
&\quad \times \Biggl(\prod_{m={j+1}}^k p_{s_m-s_{m-1}}(y_{m-1},y_m) \Biggr)p_{T-s_k}(y_k,z)dz \prod_{q=1}^k dy_q \prod_{r=1}^k ds_{r} \Bigg|,\label{etaj}
\end{align}
where $j=\overline{2,k}$.

Consider (\ref{j}) in more detail. We rewrite it in the form
\begin{align*}
&\Bigg|\int_{0}^T\int_{S_{j-3,0,s_{j-2}}}\int_{s_{j-2}}^T\int_{s_{j-2}}^{s_j}\int_{S_{k-j,s_j,T}}\int_{(\mathbb{R}^d)^{k+1}} \Biggl(\prod_{i=1}^kh(y_i) \Biggr)f(z)\\
&\qquad \times \Biggl(\prod_{l=1}^{j-1}p_{\eta_n(s_l)-\eta_n(s_{l-1})}(y_{l-1},y_l) \Biggr)\bigl(p_{s_j-s_{j-1}}(y_{j-1},y_j) - p_{s_j-\eta_n(s_{j-1})}(y_{j-1},y_j)\bigr)\\
&\qquad \times \Biggl(\prod_{m={j+1}}^k p_{s_m-s_{m-1}}(y_{m-1},y_m) \Biggr)p_{T-s_k}(y_k,z)\\
&\qquad \times dz \prod_{q=1}^k dy_q \prod_{r=j+1}^k ds_{r} ds_{j-1}ds_j \prod_{v=1}^{j-3} ds_{v} ds_{j-2} \Bigg|\\
&\quad \leq \Bigg|\int_{0}^T\int_{S_{j-3,0,s_{j-2}}}\int_{s_{j-2}}^T\int_{s_{j-2}}^{s_j-T/n}\int_{S_{k-j,s_j,T}}\int_{(\mathbb{R}^d)^{k+1}} \Biggl(\prod_{i=1}^kh(y_i) \Biggr)f(z)\\
&\qquad \times \Biggl(\prod_{l=1}^{j-1}p_{\eta_n(s_l)-\eta_n(s_{l-1})}(y_{l-1},y_l) \Biggr)\bigl(p_{s_j-s_{j-1}}(y_{j-1},y_j) - p_{s_j-\eta_n(s_{j-1})}(y_{j-1},y_j)\bigr)\\
&\qquad \times \Biggl(\prod_{m={j+1}}^k p_{s_m-s_{m-1}}(y_{m-1},y_m) \Biggr)p_{T-s_k}(y_k,z)\\
&\qquad \times dz \prod_{q=1}^k dy_q \prod_{r=j+1}^k ds_{r} ds_{j-1}ds_j \prod_{v=1}^{j-3} ds_{v} ds_{j-2} \Bigg|\\
&\qquad + \Bigg|\int_{0}^T\int_{S_{j-3,0,s_{j-2}}}\int_{s_{j-2}}^T\int_{s_j-T/n}^{s_j}\int_{S_{k-j,s_j,T}}\int_{(\mathbb{R}^d)^{k+1}} \Biggl(\prod_{i=1}^kh(y_i) \Biggr)f(z)\\
&\qquad \times \Biggl(\prod_{l=1}^{j-1}p_{\eta_n(s_l)-\eta_n(s_{l-1})}(y_{l-1},y_l) \Biggr)\bigl(p_{s_j-s_{j-1}}(y_{j-1},y_j) - p_{s_j-\eta_n(s_{j-1})}(y_{j-1},y_j)\bigr)\\
&\qquad \times \Biggl(\prod_{m={j+1}}^k p_{s_m-s_{m-1}}(y_{m-1},y_m) \Biggr)p_{T-s_k}(y_k,z)\\
&\qquad \times dz \prod_{q=1}^k dy_q \prod_{r=j+1}^k ds_{r} ds_{j-1}ds_j \prod_{v=1}^{j-3} ds_{v} ds_{j-2} \Bigg|.
\end{align*}
We estimate each term separately:
\begin{align*}
&\Bigg|\int_{0}^T\int_{S_{j-3,0,s_{j-2}}}\int_{s_{j-2}}^T\int_{s_j-T/n}^{s_j}\int_{S_{k-j,s_j,T}}\int_{(\mathbb{R}^d)^{k+1}} \Biggl(\prod_{i=1}^kh(y_i) \Biggr)f(z)\\
&\qquad \times \Biggl(\prod_{l=1}^{j-1}p_{\eta_n(s_l)-\eta_n(s_{l-1})}(y_{l-1},y_l) \Biggr)\bigl(p_{s_j-s_{j-1}}(y_{j-1},y_j) - p_{s_j-\eta_n(s_{j-1})}(y_{j-1},y_j)\bigr)\\
&\qquad \times \Biggl(\prod_{m={j+1}}^k p_{s_m-s_{m-1}}(y_{m-1},y_m) \Biggr)p_{T-s_k}(y_k,z)\\
&\qquad \times dz \prod_{q=1}^k dy_q \prod_{r=j+1}^k ds_{r} ds_{j-1}ds_j \prod_{v=1}^{j-3} ds_{v} ds_{j-2} \Bigg|\\
&\quad \leq\|h\|^k\|f\|\\
&\qquad \times\int_{0}^T\int_{S_{j-3,0,s_{j-2}}}\int_{s_{j-2}}^T\int_{s_j-T/n}^{s_j}\int_{S_{k-j,s_j,T}} \int_{(\mathbb{R}^d)^{j}} \Biggl(\prod_{l=1}^{j-1} p_{\eta_n(s_l)-\eta_n(s_{l-1})}(y_{l-1},y_l)\Biggr)\\
&\qquad \times\big|p_{s_j-s_{j-1}}(y_{j-1},y_j) - p_{s_j-\eta_n(s_{j-1})}(y_{j-1},y_j)\big|\\
&\qquad \times\prod_{q=1}^j dy_q \prod_{r=j+1}^k ds_{r}ds_{j-1}ds_j \prod_{v=1}^{j-3}ds_{v} ds_{j-2}\\
&\quad \leq\frac{2}{(k-1)!} \|h\|^k\|f\| T^{k} \frac{1}{n}.
\end{align*}
For the other term, we have
\begin{align*}
&\Bigg|\int_{0}^T\int_{S_{j-3,0,s_{j-2}}}\int_{s_{j-2}}^T\int_{s_{j-2}}^{s_j-T/n}\int_{S_{k-j,s_j,T}}\int_{(\mathbb{R}^d)^{k+1}} \Biggl(\prod_{i=1}^kh(y_i) \Biggr)f(z)\\
&\qquad \times \Biggl(\prod_{l=1}^{j-1}p_{\eta_n(s_l)-\eta_n(s_{l-1})}(y_{l-1},y_l) \Biggr)\bigl(p_{s_j-s_{j-1}}(y_{j-1},y_j) - p_{s_j-\eta_n(s_{j-1})}(y_{j-1},y_j)\bigr)\\
&\qquad \times \Biggl(\prod_{m={j+1}}^k p_{s_m-s_{m-1}}(y_{m-1},y_m)\Biggr)p_{T-s_k}(y_k,z)\\
&\qquad \times dz \prod_{q=1}^k dy_q \prod_{r=j+1}^k ds_{r} ds_{j-1}ds_j \prod_{v=1}^{j-3} ds_{v} ds_{j-2}\Bigg|\\
&\quad = \Bigg|\int_{0}^T\int_{S_{j-3,0,s_{j-2}}}\int_{s_{j-2}}^T\int_{s_{j-2}}^{s_j-T/n}\int_{\eta_n(s_{j-1})}^{s_{j-1}} \int_{S_{k-j,s_j,T}}\int_{(\mathbb{R}^d)^{k+1}} \Biggl(\prod_{i=1}^kh(y_i)\Biggr)f(z)\\
&\qquad \times \Biggl(\prod_{l=1}^{j-1}p_{\eta_n(s_l)-\eta_n(s_{l-1})}(y_{l-1},y_l) \Biggr)\partial_{s_j-u}p_{s_j-u}(y_{j-1},y_j)\\
&\qquad \times \Biggl(\prod_{m={j+1}}^k p_{s_m-s_{m-1}}(y_{m-1},y_m) \Biggr) p_{T-s_k}(y_k,z)\\
&\qquad \times dz \prod_{q=1}^k dy_q \prod_{r=j+1}^k ds_{r} du ds_{j-1} ds_j \prod_{v=1}^{j-3} ds_{v}ds_{j-2} \Bigg|\\[1pt]
&\quad \leq\|h\|^k\|f\|\\[1pt]
&\qquad \times\int_{0}^T\!\!\int_{S_{j-3,0,s_{j-2}}}\int_{s_{j-2}}^T\int_{s_{j-2}}^{s_j-T/n}\!\int_{\eta_n(s_{j-1})}^{s_{j-1}} \int_{S_{k-j,s_j,T}} \int_{(\mathbb{R}^d)^{j}} \big|\partial _{s_j-u}p_{s_j-u}(y_{j-1},y_j)\big|\\
&\qquad \times \Biggl(\prod_{l=1}^{j-1}p_{\eta_n(s_l)-\eta_n(s_{l-1})}(y_{l-1},y_l) \Biggr) \prod_{q=1}^j dy_q \prod_{r=j+1}^k ds_{r} du ds_{j-1}ds_j \prod_{v=1}^{j-3}ds_{v}ds_{j-2}\\[1pt]
&\quad \leq\|h\|^k\|f\|\\[1pt]
&\qquad \times\int_{0}^T \!\!\int_{S_{j-3,0,s_{j-2}}}\int_{s_{j-2}}^T \int_{S_{k-j,s_j,T}} \int_{0}^{s_j-T/n}\! \int_{\eta_n(s_{j-1})}^{s_{j-1}}\int_{(\mathbb{R}^d)^{j}} \big|\partial_{s_j-u}p_{s_j-u}(y_{j-1},y_j)\big|\\
&\qquad \times \Biggl(\prod_{l=1}^{j-1}p_{\eta_n(s_l)-\eta_n(s_{l-1})}(y_{l-1},y_l) \Biggr) \prod_{q=1}^j dy_q du ds_{j-1}\prod_{r=j+1}^k ds_{r}ds_j \prod_{v=1}^{j-3}ds_{r} ds_{j-2}.
\end{align*}
Again, we consider the inner integral:
\begin{align*}
&\int_{0}^{\eta_n(s_j)-T/n} \int_{\eta_n(s_{j-1})}^{s_{j-1}}\int_{(\mathbb{R}^d)^2} p_{\eta_n(s_{j-1})-\eta_n(s_{j-2})}(y_{j-2},y_{j-1})\\[1pt]
&\qquad \times\big|\partial_{s_j-u}p_{s_j-u}(y_{j-1},y_j)\big|dy_j dy_{j-1} du ds_{j-1}\\[1pt]
&\quad =\sum_{i=0}^{\eta_n(s_j)n/T-2} \int_{iT/n}^{(i+1)T/n}\int_{iT/n}^{s_{j-1}} \int_{(\mathbb{R}^d)^2}p_{iT/n-\eta_n(s_{j-2})}(y_{j-2},y_{j-1})\\[1pt]
&\qquad \times\big|\partial_{s_j-u}p_{s_j-u}(y_{j-1},y_j)\big|dy_j dy_{j-1} du ds_{j-1}\\[1pt]
&\quad =\sum_{i=0}^{\eta_n(s_j)n/T-2} \int_{iT/n}^{(i+1)T/n}\int_{u}^{(i+1)T/n} \int_{(\mathbb{R}^d)^2}p_{iT/n-\eta_n(s_{j-2})}(y_{j-2},y_{j-1})\\[1pt]
&\qquad \times\big|\partial_{s_j-u}p_{s_j-u}(y_{j-1},y_j)\big|dy_j dy_{j-1} ds_{j-1}du\\[1pt]
&\quad \leq\frac{T}{n} \sum_{i=0}^{\eta_n(s_j)n/T-2}\int_{iT/n}^{(i+1)T/n} \int_{(\mathbb{R}^d)^2}p_{iT/n-\eta_n(s_{j-2})}(y_{j-2},y_{j-1})\\[1pt]
&\qquad \times\big|\partial_{s_j-u}p_{s_j-u}(y_{j-1},y_j)\big|dy_j dy_{j-1} du\\[1pt]
&\quad \leq C_T \frac{T}{n} \sum_{i=0}^{\eta_n(s_j)n/T-2}\int_{iT/n}^{(i+1)T/n} \int_{\mathbb{R}^d}p_{iT/n-\eta_n(s_{j-2})}(y_{j-2},y_{j-1}) (s_j-u)^{-1}dy_{j-1} du\\[1pt]
&\quad = C_T \frac{T}{n} \sum_{i=0}^{\eta_n(s_j)n/T-2}\int_{iT/n}^{(i+1)T/n} (s_j-u)^{-1} du = C_T \frac{T}{n} \int_{0}^{\eta_n(s_j)-T/n}(s_j-u)^{-1} du.
\end{align*}
We have
\begin{align*}
&\int_{0}^{s_j-T/n} \int_{\eta_n(s_{j-1})}^{s_{j-1}}\int_{(\mathbb{R}^d)^2} p_{\eta_n(s_{j-1})-\eta_n(s_{j-2})}(y_{j-2},y_{j-1})\\
&\qquad \times\big|\partial_{s_j-u}p_{s_j-u}(y_{j-1},y_j)\big|dy_j dy_{j-1} du ds_{j-1}\\
&\quad \leq C_T \frac{T}{n} \int_{0}^{s_j-T/n}(s_j-u)^{-1} du \leq TC_T \frac{\log n}{n}.
\end{align*}
Therefore, we obtain
\begin{align*}
&\Bigg|\int_{S_{k,0,T}}\int_{(\mathbb{R}^d)^{k+1}} \Biggl(\prod_{i=1}^kh(y_i) \Biggr)f(z)\Biggl(\prod_{l=1}^{j-1} p_{\eta_n(s_l)-\eta_n(s_{l-1})}(y_{l-1},y_l)\Biggr)\\
&\qquad \times\bigl(p_{s_j-s_{j-1}}(y_{j-1},y_j) - p_{s_j-\eta_n(s_{j-1})}(y_{j-1},y_j)\bigr) \Biggl(\prod_{m={j+1}}^k p_{s_m-s_{m-1}}(y_{m-1},y_m)\Biggr)\\
&\qquad \times p_{T-s_k}(y_k,z) dz \prod_{q=1}^k dy_q \prod_{r=1}^k ds_{r} \Bigg|\\
&\quad \leq3 C_T\frac{1}{(k-1)!} \|h\|^k\|f\| T^{k} \frac{\log n}{n}.
\end{align*}
Analogously, we also have:
\begin{align*}
&\Bigg|\int_{S_{k,0,T}}\int_{(\mathbb{R}^d)^{k+1}} \Biggl(\prod_{i=1}^kh(y_i) \Biggr)f(z)\Biggl(\prod_{l=1}^{j-1} p_{\eta_n(s_l)-\eta_n(s_{l-1})}(y_{l-1},y_l)\Biggr)\\
&\qquad \times\bigl(p_{s_j-\eta_n(s_{j-1})}(y_{j-1},y_j) - p_{\eta_n(s_j)-\eta_n(s_{j-1})}(y_{j-1},y_j)\bigr)\\
&\qquad \times \Biggl(\prod_{m={j+1}}^k p_{s_m-s_{m-1}}(y_{m-1},y_m) \Biggr)p_{T-s_k}(y_k,z)dz \prod_{q=1}^k dy_q \prod_{r=1}^k ds_{r} \Bigg|\\
&\quad \leq3 C_T \frac{1}{(k-1)!} \|h\|^k\|f\|T^{k} \frac{\log n}{n}.
\end{align*}
Therefore, we finally obtain
\[
\big|E_x \bigl[\bigl(I_{T}(h)\bigr)^k-
\bigl(I_{T,n}(h)\bigr)^k \bigr] f(X_T) \big|
\leq6k^2 C_T T^{k} \|h\|^k \biggl(
{\log{n}\over n} \biggr) \|f\|,
\]
which completes the proof.
\qed

\section{Applications}\label{s3}

\subsection{Discrete approximation of the Feynman--Kac semigroup}

Let $X$ be a Brownian motion with values in $\mathbb{R}^d$. Then
condition \textbf{X} holds with
\[
q_{t,x}(y) = c_1 t^{-d/2} \exp
\bigl(-c_2t^{-1}|x-y|^2\bigr),
\]
where $c_1, c_2$ are some positive constants.

Let $h$ be a bounded measurable function. Then, it is known (see, e.g.,
\cite{Sznitman}, Chapter~1) that the family of operators
\[
R_t^h f(x) = E_x \bigl[f(X_t)
\exp \bigl\{\lambda I_t(h) \bigr\} \bigr]
\]
forms a semigroup on $L_p(\mathbb{R}^d), \ p\geq1$, and its generator equals
\[
\mathcal{A}_h f= {1\over2}\Delta f + \lambda h f.
\]
This semigroup is called the Feynman--Kac semigroup.

Denote
\[
R_{t,n}^h f(x) = E_x \bigl[f(X_t)
\exp \bigl\{\lambda I_{t,n}(h) \bigr\} \bigr].
\]
Then, using the Taylor expansion of the exponential function and
Theorem~\ref{thm1}, we have the following statement.
\begin{cor}\label{n2}
For any bounded functions $f,h$ and real positive number $\lambda$, we have:
\[
\big|R_{t}^h f(x) - R_{t,n}^h f(x) \big|
\leq C_{T,\lambda,h} \biggl({\log{n}\over n} \biggr) \|f\|,
\]
where
\[
C_{T,\lambda,h}=6C_T\lambda\|h\|T \big(1+\lambda\|h\|T\big) \exp\big\{\lambda
\|h\| T\big\}.
\]
\end{cor}

Therefore, the main result of this paper provides an approximation of
the\break Feynman--Kac semigroup with accuracy $(\log{n})/ n$.

\subsection{Approximation of the price of an occupation-time option}

Let the price of an asset $S=\{S_t, t\geq0\}$ be of the form
\[
S_t = S_0 \exp(X_t),
\]
where $X$ is a one-dimensional Markov process satisfying condition
\textbf{X}.
The time spent by $S$ in a defined set $J\subset\mathbb{R}$ (or the
time spent by $X$ in a set $J'=\{x:e^x\in J\}$) from time 0 to time $T$
is given by
\[
\int_0^T \mathbb{I}_{\{S_t \in J\}} dt = \int
_0^T \mathbb{I}_{\{X_t \in
J'\}} dt.
\]

We consider an occupation-time option (see \cite{Linetsky}) whose price
depends on the time spent by the process $S$ in a set $J$. In contrast
to the traditional barrier options, which are activated or canceled
when the process $S$ hits a defined level (barrier), the payoff of an
occupation-time option depends on the time spent by the price of the
asset above/below this level.

For the strike price $K$, the barrier $L$, and the knock-out rate $\rho
$, the payoff of a down-and-out call occupation-time option is given by
\[
\exp \Biggl( - \rho\int_0^T
\mathbb{I}_{\{S_t \leq L\}} dt \Biggr) (S_T - K)_+.
\]
Then, for the risk-free interest rate $r$, its price is given by
\[
\textbf{C}(T) = \exp(-rT) E \Biggl[ \exp \Biggl( - \rho\int_0^T
\mathbb {I}_{\{S_t \leq L\}} dt \Biggr) (S_T - K)_+ \Biggr].
\]

Denote
\[
\textbf{C}_n(T) = \exp(-rT) E \Biggl[ \exp \Biggl( - \rho T/n \sum
_{k=0}^{n-1} \mathbb{I}_{\{S_{kT/n} \leq L\}} dt
\Biggr) (S_T - K)_+ \Biggr].
\]
We provide the following corollary of Theorem~\ref{thm1}.
\begin{prop}
Suppose that \textup{\textbf{X}} holds and there exists $u>1$ such that $G:=\break E
\exp(u X_T) = E S_T^u < + \infty$. Then
\[
\big|\textup{\textbf{C}}_n(T)- \textup{\textbf{C}}(T) \big| \leq3\max\{C_{T,\rho}, G
\} \exp(-rT) \biggl({\log{n}\over n^{1-1/u}} \biggr),
\]
where $C_{T,\rho}=6C_T\rho T (1+\rho T) \exp(\rho T)$.
\end{prop}

\begin{proof}

For some $N>0$, we denote
\[
\textbf{C}^N(T) = \exp(-rT) E \Biggl[ \exp \Biggl( - \rho\int
_0^T \mathbb{I}_{\{S_t \leq L\}} dt \Biggr)
\bigl((S_T - K)_+ \wedge N\bigr) \Biggr],
\]
\[
\textbf{C}^N_n(T) = \exp(-rT) E \Biggl[ \exp \Biggl( -
\rho T/n \sum_{k=0}^{n-1}
\mathbb{I}_{\{S_{kT/n} \leq L\}} dt \Biggr) \bigl((S_T - K)_+ \wedge N\bigr)
\Biggr].
\]
Then
\[
\big|\textbf{C}_n(T)- \textbf{C}(T) \big| \leq \big|\textbf{C}^N_n(T)-
\textbf{C}^N(T) \big| + \big|\textbf{C}(T)- \textbf{C}^N(T) \big| +
\big|\textbf{C}_n(T)- \textbf{C}^N_n(T) \big|.
\]
We estimate each term separately. According to Corollary~\ref{n2},
\[
\big|\textbf{C}^N_n(T)- \textbf{C}^N(T) \big|
\leq NC_{T,\rho} \exp (-rT) \biggl({\log{n}\over n} \biggr).
\]
For other terms, we have:
\begin{align*}
&\big|\textbf{C}(T)- \textbf{C}^N(T) \big| + \big|\textbf{C}_n(T)-\textbf {C}^N_n(T) \big|\\
&\quad \leq2 \exp(-rT) E \bigl[(S_T - K)_+ -(S_T - K)_+ \wedge N \bigr] \leq 2 \exp(-rT) E [S_T \mathbb{I}_{\{S_T>N\}}]\\
&\quad = 2 \exp(-rT) E \biggl[ \frac{S_T N^{u-1}\mathbb{I}_{\{S_T>N\}}}{N^{u-1}} \biggr] \leq\frac{2G}{N^{u-1}}\exp(-rT).
\end{align*}

Now, putting $N = n^{1/u}$ completes the proof.
\end{proof}

Therefore, the main result of this paper provides the approximate value
$\textbf{C}_n(T)$ of the price of an occupation-time option $\textbf
{C}(T)$ with accuracy of order $(\log n) / n^{1-1/u}$ for the class of
processes $X$ satisfying \textbf{X} and the condition $E \exp(u X_T) <
+ \infty$ for some $u>1$.

\section*{Acknowledgments}

The authors are deeply grateful to Arturo Kohatsu-Higa for discussion
and valuable suggestions about the possible area of applications of the
main result of the paper.

\bibliographystyle{vmsta-mathphys}
%

\end{document}